\documentclass[12pt]{article}

\usepackage[english]{babel}
\usepackage{amsmath}
\usepackage{amssymb,amsfonts,amsthm}
\usepackage{tikz,color,pgf,graphicx,url}
\usetikzlibrary{calc}
\usetikzlibrary{decorations.pathreplacing}
\usepackage{enumerate}     
\usepackage[T1]{fontenc}
\usepackage{xcolor} 
\usepackage{txfonts} 
\usepackage{hyperref}
\hypersetup{
	colorlinks   = true, 
	urlcolor     = black, 
	linkcolor    = black, 
	citecolor   = black 
}

\newtheorem{theorem}{Theorem}[section]

\newtheorem{lemma}[theorem]{Lemma}
\newtheorem{proposition}[theorem]{Proposition}
\newtheorem{conjecture}[theorem]{Conjecture}

\newtheorem{remark}[theorem]{Remark}

\newtheorem{claim}{Claim}

\newcommand{\mrho}{\gamma_2^2}

\newcommand{\smallqed}{{\tiny ($\Box$)}}
\DeclareMathOperator{\diam}{diam}
\DeclareMathOperator{\rad}{rad}
\DeclareMathOperator{\ecc}{ecc}
\newcommand{\F}{\mathcal{F}}

\newcommand{\vertex}{\node[vertex]}
\tikzstyle{vertex}=[circle, draw, inner sep=0pt, minimum size=6pt]

\begin{document}

\title{The $d$-distance $p$-packing domination number: complexity, cycles, and trees}

\date{}

\author{
Csilla Bujt\'as $^{a,b,}$\thanks{Email: \texttt{csilla.bujtas@fmf.uni-lj.si}}
\and
Vesna Ir\v si\v c Chenoweth $^{a,b,}$\thanks{Email: \texttt{vesna.irsic@fmf.uni-lj.si}}
\and
Sandi Klav\v zar $^{a,b,c,}$\thanks{Email: \texttt{sandi.klavzar@fmf.uni-lj.si}}
\and
Gang Zhang $^{b,d,}$\thanks{Email: \texttt{gzhang@stu.xmu.edu.cn}}
}

\maketitle

\begin{center}
$^a$ Faculty of Mathematics and Physics, University of Ljubljana, Slovenia\\
\medskip

$^b$ Institute of Mathematics, Physics and Mechanics, Ljubljana, Slovenia\\
\medskip

$^c$ Faculty of Natural Sciences and Mathematics, University of Maribor, Slovenia\\
\medskip

$^d$ School of Mathematical Sciences, Xiamen University, China\\
\medskip
\end{center}

\begin{abstract}
A set of vertices $X\subseteq V(G)$ is a $d$-distance dominating set if for every $u\in V(G)\setminus X$ there exists $x\in X$ such that $d(u,x) \le d$, and $X$ is a $p$-packing if $d(u,v) \ge p+1$ for every different $u,v\in X$. The $d$-distance $p$-packing domination number $\gamma_d^p(G)$ of $G$ is the minimum size of a set of vertices of $G$ which is both a $d$-distance dominating set and a $p$-packing. It is proved that for every two fixed integers $d$ and $p$ with $2 \le d$ and $0 \le p \leq 2d-1$, the decision problem whether $\gamma_d^p(G) \leq k$ holds is NP-complete for bipartite planar graphs. A necessary and sufficient condition for the existence of a $d$-distance $p$-packing dominating set in $C_n$ is obtained and $\gamma_d^p(C_n)$ determined for every $d$, $p$, and $n$. For a tree $T$ on $n$ vertices with $\ell$ leaves and $s$ support vertices it is proved that (i) $\gamma_2^0(T) \geq \frac{n-\ell-s+4}{5}$, (ii) $\left \lceil \frac{n-\ell-s+4}{5} \right \rceil \leq \gamma_2^2(T) \leq \left \lfloor \frac{n+3s-1}{5} \right \rfloor$, and if $d \geq 2$, then (iii) $\gamma_d^2(T) \leq \frac{n-2\sqrt{n}+d+1}{d}$. Inequality (i) improves an earlier bound due to Meierling and Volkmann, and independently Raczek, Lema\'nska, and Cyman, while (iii) extends an earlier result for $\gamma_2^2(T)$ due to Henning. Sharpness of the bounds are discussed and established in most cases. It is also proved that every connected graph $G$ contains a spanning tree $T$ such that $\gamma_2^2(T) \leq \gamma_2^2(G)$.
\end{abstract}


\medskip\noindent
{\bf AMS Subj.\ Class.\ (2020)}: 05C69, 05C05, 68R10

\section{Introduction}
\label{sec:intro}

Let $G = (V(G), E(G))$ be a graph, and let $d(u,v)$ denote the shortest-path distance in $G$ between vertices $u, v\in V(G)$. Let further $d$ and $p$ be nonnegative integers and $X\subseteq V(G)$. Then $X$ is a {\em $d$-distance dominating set} of $G$ if for every vertex $u\in V(G)\setminus X$ there exists a vertex $x\in X$ such that $d(u,x) \le d$, and $X$ is a {\em $p$-packing} of $G$ if $d(u,v) \ge p+1$ for every different vertices $u,v\in X$. (Note that the set $X$ is a $1$-packing if and only if $X$ is an independent set.) The \emph{ $d$-distance $p$-packing domination number} $\gamma_d^p(G)$ of $G$ is the minimum size of a set of vertices of $G$ which is both $d$-distance dominating set and $p$-packing. If $d <p$, such a set might not exist, in which case we set $\gamma_d^p(G) = \infty$. 

The $d$-distance $p$-packing domination number was introduced back in 1994 by Beineke and Henning~\cite{Beineke-1994} under the name {\em $(p,d)$-domination number} and with the notation $i_{p,d}(G)$. However, as this, with the exception of~\cite{Gimbel-1996}, has not been used subsequently, we have opted for the above terminology and notation with the intention of placing this general concept within the trends of contemporary graph domination theory. We now describe a number of important special cases.

\begin{itemize}
\item $\gamma_1^0(G)$ is the usual domination number $\gamma(G)$ of $G$. More generally, $\gamma_d^0(G)$ is studied in the literature under the name of {\em $d$-distance domination number} of $G$, often denoted by $\gamma_d(G)$, see the survey~\cite{henning-2020}.
\item $\gamma_1^1(G)$ is the {\em independent domination number}  $i(G)$ of $G$, which is one of the core concepts in domination theory~\cite{Haynes-2023}. More generally, $\gamma_d^1(G)$, is the {\em $d$-distance independent domination number} of $G$, denoted by $id(d,G)$ in~\cite{Gimbel-1996, henning-2020}. 
\item $\gamma_2^2(G)$ is the {\em lower packing number of $G$}, which can be equivalently described as the minimum cardinality of a maximal $2$-packing of $G$, see~\cite{Henning-1998-1}, where it is denoted by $\rho_L(G)$. More generally, $\gamma_d^d(G)$ has been investigated for the first time by Henning, Oellermann, and Swart in~\cite{Henning1991} under the name {\em $d$-independent $d$-domination number} and denoted by $i_d(G)$. In this article, we will pay considerable attention to the $d$-distance $2$-packing domination number (of trees), hence we refer to the following selected papers~\cite{Bozovic-2022, Fisher-1994, Junosza-2012, Trejo-2023} that deal with different aspects of $2$-packings.  
\end{itemize}
Note that $\gamma_0^0(G) = |V(G)|$ and that if $p\ge 1$, and $G$ is connected with at least two vertices, then $\gamma_0^p(G) = \infty$. In general, for every $d \ge 0$ and $p \ge 2d+1$, if $G$ is a connected graph with a radius $\rad(G) >d$, then $\gamma_d^p(G) = \infty$. The folloing result also follows from the definition. 

\begin{proposition}
\label{prop:lower-bound-k-domination}
If $0 \leq d'\leq d$ and $0\le p\le p'$, then $\gamma_{d'}^{p'}(G) \geq \gamma_d^p(G)$.
\end{proposition}

\subsection{Notation}

In this brief subsection, we collect additional definitions needed.  

Let $G = (V(G), E(G))$ be a graph and $k \geq 0$ an integer. The \emph{neighborhood} $N(v)$ of a vertex $v \in V(G)$ contains the neighbors of $v$. The \emph{closed neighborhood} $N[v]$ of $v$ contains $v$ and its neighbors. We also define 
$$N_k[v]= \{u \in V(G): d(u,v) \leq k\}\,.$$ 
The \emph{eccentricity} of a vertex $v \in V(G)$ is $\ecc(v) = \max \{ d(v, x): x \in V(G) \}$. The \emph{radius} $\rad(G)$ and the \emph{diameter} $\diam(G)$ of $G$ are, respectively, the minimum and the maximum eccentricity among vertices of $G$. If $d\ge 1$, then a set $X\subseteq V(G)$ is a $d$-perfect code if for every $u\in V(G)$ there exists a unique $x\in V(G)$ such that $u\in N_d[x]$, cf.~\cite{taylor-2009}. We say that a set $D\subseteq V(G)$ is a $\gamma_d^p$-set of $G$ if $D$ is a $d$-distance dominating set and a $p$-packing with $|D| = \gamma_d^p(G)$.

Let $T$ be a tree. A vertex of degree 1 in $T$ is called a \emph{leaf} and its neighbor is called a \emph{support vertex}; let $S(T)$ denote the set of the support vertices of  $T$. If $uv \in E(T)$, then $T-uv$ has two components; the one containing $u$ is denoted with $T_u$ while the one containing $v$ is $T_v$. When it will be clear from the context, $n_x$, $\ell_x$, and $s_x$ will denote the number of vertices, the number of leaves, and the number of support vertices in $T_x$, respectively. 

\subsection{Our results}

In Section~\ref{sec:complexity} we prove that for every two fixed integers $d$ and $p$ with $2 \le d$ and $0 \le p \leq 2d-1$, the decision problem whether $\gamma_d^p(G) \leq k$ holds is NP-complete over the class of bipartite planar graphs. Further, if $p=2d$, then the problem is NP-complete over the class of planar graphs.  (For the remaining values of $d$ and $p$, the algorithmic time complexity is known.) The NP-hardness over the class of bipartite graphs was already established in~\cite{Fricke} for $d=p \ge 3$, but in the case of $d=p=2$ the proof contains a mistake that we correct here. The main result of Section~\ref{sec:cycles} establishes  the exact value of $\gamma_d^p(C_n)$ for every $d$, $p$, and $n$. It includes a necessary and sufficient condition for the existence of a $d$-distance $p$-packing dominating set in $C_n$. The exact value of $\gamma_d^p(P_n)$ is also determined in all cases.

In the subsequent sections, we focus on trees. Let ${\cal T}_d$ be the set of trees in which leaves are pairwise at distance $2d\ (\bmod\ {2d+1})$, that is, 
$${\cal T}_d = \{ T:\ T \text{ tree},\ d(x,y) \equiv 2d\ (\bmod\ {2d+1}) \text{ for every different leaves } x, y \text{ in } T\}\,.$$
Meierling and Volkmann~\cite{Meierling-2005}, and independently Raczek, Lema\'nska, and Cyman~\cite{Raczek-2006} proved the following result, extending the earlier result for $d = 1$ from~\cite{Lemanska-2004}.

\begin{theorem} {\rm \cite{Meierling-2005, Raczek-2006}}
\label{thm:k-distance-in-trees}
If $d\ge 1$ and $T$ is a tree on $n$ vertices and with $\ell$ leaves, then 
$$\gamma_d^0(T) \geq \frac{n - d \ell + 2d}{2d+1}\,.$$
Moreover, equality holds if and only if $T\in {\cal T}_d$.
\end{theorem}
Setting 
$$\F_2 = \{ T:\ T \text{ tree},\ d(x,y) \equiv 2\ (\bmod\ {5}) \text{ for every } x, y \in S(T), x\ne y\} \setminus \{K_{1,n}:\ n \geq 2\}\,,$$
we improve Theorem~\ref{thm:k-distance-in-trees} for the case $d=2$ as follows. 

\begin{theorem}
    \label{thm:lower-bound-trees-2-domination}
    If $T$ is a tree on $n$ vertices with $\ell$ leaves and $s$ support vertices, then $$\gamma_2^0(T) \geq \frac{n-\ell-s+4}{5}\,.$$ 
    Moreover, equality holds if and only if $T\in \F_2$.
\end{theorem}

We then bound the $2$-distance $2$-packing domination number of trees as follows.

\begin{theorem}
    \label{thm:2-2-lower-upper}
    If $T$ is a tree on $n \ge 2$ vertices with $\ell$ leaves and $s$ support vertices, then 
    $$\left \lceil \frac{n-\ell-s+4}{5} \right \rceil \leq \mrho(T) \leq \left \lfloor \frac{n+3s-1}{5} \right \rfloor\,.$$
\end{theorem}

Theorems~\ref{thm:lower-bound-trees-2-domination} and~\ref{thm:2-2-lower-upper} are respectively proved in Sections~\ref{sec:proof1} and~\ref{sec:proof2}. 

In \cite{Gimbel-1996}, Gimbel and Henning proved that if $G$ is a connected graph of order $n \geq d+1$, then $\gamma_d^1(G) \leq \frac{n-2\sqrt{n}+d+1}{d}$ and the bound is sharp. In \cite{Henning-1998-1}, Henning proved that $\mrho(T) \leq \frac{n-2\sqrt{n}+3}{2}$ for all trees $T$ of order $n \geq 3$ and that the bound is sharp. The following theorem generalizes Henning's result to all $d \geq 2$, as well as partially generalizes the result of \cite{Gimbel-1996}.

\begin{theorem}
\label{thm:trees-upper-bound-gamma_d^2}
If $d \geq 2$ and $T$ is a tree of order $n$, then 
$$\gamma_d^2(T) \leq \frac{n-2\sqrt{n}+d+1}{d}\,.$$
\end{theorem}

Theorem~\ref{thm:trees-upper-bound-gamma_d^2} is proved in Section~\ref{sec:proof3}, where it is in addition demonstrated that its bound is best possible. 

Finally, in Section~\ref{sec:spanning-trees} we prove that every connected graph $G$ contains a spanning tree $T$ such that $\mrho(T) \leq \mrho(G)$.

\section{Algorithmic time complexity}
\label{sec:complexity}

In this section, we prove the NP-completeness of the decision problem whether $\gamma_d^p(G) \leq k$ holds. Our results show that the NP-hardness is true for every two fixed integers $d$ and $p$ with $2 \le d$ and $0 \le p \leq 2d-1$ over the class of bipartite planar graphs, and also if $p=2d$ and the problem is considered for planar graphs. For the remaining values of $d$ and $p$, the algorithmic time complexity is known. If $p \ge 2d+1$, a $d$-distance $p$-packing dominating set exists in $G$ if and only if its radius is at most $d$, and then $\gamma_d^p (G) =1$ holds. The decision problem, therefore, can be solved in polynomial time if $p \ge 2d+1$. If $d=1$ and $p \in \{0,1,2\}$, then the problem is known to be NP-complete as $\gamma_d^p(G)$ corresponds to the domination number, independent domination number, and the minimum size of a $1$-perfect code in $G$. We also note that the NP-hardness over the class of bipartite graphs was already established in~\cite{Fricke} for $d=p \ge 3$. However, the proof in~\cite{Fricke} for $d=p=2$ contains a mistake that we correct here by applying a reduction from the $1$-in-$3$-SAT problem in Case~$2$ when proving Theorem~\ref{thm:bipartite-NP}. 
We first define the problems and present the constructions used in the proofs.

\paragraph{Planar 3-SAT and Planar 1-in-3-SAT problems.}
A formula $F$ is an instance of these problems if the following properties hold.
\begin{itemize}
    \item[(i)] $F$ is a 3-SAT instance $F= C_1 \wedge \cdots \wedge C_\ell$ over the Boolean variables $x_1, \dots, x_k$ and hence each clause $C_i$ is a disjunction of three literals.
    \item[(ii)] Given a 3-SAT formula $F$, we consider the graph $G^*(F)$ associated with $F$. The vertex set of $G^*(F)$ contains one vertex $c_j$ for each clause $C_j$ (clause vertices) and two vertices $x_i^+$ and $x_i^-$ for every variable $x_i$ (literal vertices). We add the edge $x_i^+x_i^-$ for every $i \in [k]$; an edge between $c_j$ and $x_i^+$ if clause $C_j$ contains the positive literal $x_i$; and an edge between $x_i^-$ and $c_j$ if $C_j$ contains the negative literal $\bar{x_i}$. If $F$ is an instance of the Planar 3-SAT or the Planar 1-in-3-SAT problem, then $G^*(F)$ is required to be a planar graph.
\end{itemize}
We say that formula $F$ is 1-in-3-satisfiable if there is a truth assignment $$\phi: \{x_1, \dots, x_k\} \rightarrow \{ \mbox{true, false}\}$$ 
such that every clause is satisfied by exactly one literal. Further, formula $F$ is satisfied by $\phi$ if every clause is satisfied by at least one literal.
\paragraph{Construction of $G_d(F)$.}  For every $d \ge 2$ and an instance $F$ of Planar 3-SAT, we modify $G^*(F)$ to obtain $G_d(F)$ as follows. For every $i \in [k]$, we take $2d$ new vertices $y_i^1, \dots , y_i^{d}, z_i^1, \dots, z_i^{d} $, such that $x_i^+y_i^1\dots y_i^d$ and $x_i^- z_i^1 \dots z_i^d$ are paths of length $d$, remove the edge $x_i^+ x_i^-$, and add the edges $x_i^+z_i^1$ and $x_i^-y_i^1$. Let $X_i$ denote the set of these $2d+2$ vertices, for $i \in [k]$.
Further, every edge $x_i^+c_j$ or $x_i^-c_j$ between a clause and a literal vertex is subdivided by $d-2$ vertices to obtain the path $P_{i,j}^+$ or $P_{i,j}^-$, respectively. We note that $G_d(F)$ is a planar bipartite graph when $F$ is a Planar 3-SAT (or Planar 1-in-3-SAT) instance. For an example see Fig.~\ref{fig:G4(F)}. 

\begin{figure}[ht!]
\begin{center}
\begin{tikzpicture}[scale=0.9,style=thick,x=1cm,y=1cm]
\def\vr{3pt}
\begin{scope}[xshift=-0cm, yshift=0cm] 
\coordinate(x1+) at (0.0,0.0);
\coordinate(y11) at (0.0,1.0);
\coordinate(y12) at (0.0,2.0);
\coordinate(y13) at (0.0,3.0);
\coordinate(y14) at (0.0,4.0);
\coordinate(x1-) at (1.5,0.0);
\coordinate(z11) at (1.5,1.0);
\coordinate(z12) at (1.5,2.0);
\coordinate(z13) at (1.5,3.0);
\coordinate(z14) at (1.5,4.0);
\coordinate(c1) at (3.5,-3.0);
\coordinate(c2) at (8.5,-3.0);
\coordinate(x11) at (1.17,-1.0);
\coordinate(x12) at (2.33,-2.0);
\coordinate(x21) at (4.0,-2.0);
\coordinate(x22) at (4.5,-1.0);
\coordinate(x31) at (5.17,-2.0);
\coordinate(x32) at (6.83,-1.0);
\coordinate(yy11) at (6.17,-2.0);
\coordinate(yy12) at (3.84,-1.0);
\coordinate(yy21) at (7.33,-2.0);
\coordinate(yy22) at (6.16,-1.0);
\coordinate(yy31) at (9.17,-2.0);
\coordinate(yy32) at (9.83,-1.0);
\draw (x1+) -- (y14);  
\draw (x1-) -- (z14);  
\draw (x1+) -- (z11);  
\draw (x1-) -- (y11);  
\draw (c1) -- (0,0);  
\draw (c1) -- (5,0);
\draw (c1) -- (8.5,0);
\draw (c2) -- (1.5,0); 
\draw (c2) -- (5,0);
\draw (c2) -- (10.5,0);
\draw(x1+)[fill=white] circle(\vr);
\draw(x1-)[fill=white] circle(\vr);
\draw(c1)[fill=white] circle(\vr);
\draw(c2)[fill=white] circle(\vr);
\draw(x11)[fill=white] circle(\vr);
\draw(x12)[fill=white] circle(\vr);
\draw(x21)[fill=white] circle(\vr);
\draw(x22)[fill=white] circle(\vr);
\draw(x31)[fill=white] circle(\vr);
\draw(x32)[fill=white] circle(\vr);
\draw(yy11)[fill=white] circle(\vr);
\draw(yy12)[fill=white] circle(\vr);
\draw(yy21)[fill=white] circle(\vr);
\draw(yy22)[fill=white] circle(\vr);
\draw(yy31)[fill=white] circle(\vr);
\draw(yy32)[fill=white] circle(\vr);
\foreach \i in {1,2,3,4} 
{
\draw(y1\i)[fill=white] circle(\vr);
\draw(z1\i)[fill=white] circle(\vr);
}
\node at (-0.4,0.3) {$x_1^{+}$};
\node at (1.9,0.2) {$x_1^{-}$};
\node at (-0.4,1.2) {$y_1^{1}$};
\node at (-0.4,2.2) {$y_1^{2}$};
\node at (-0.4,3.2) {$y_1^{3}$};
\node at (-0.4,4.2) {$y_1^{4}$};
\node at (1.9,1.2) {$z_1^{1}$};
\node at (1.9,2.2) {$z_1^{2}$};
\node at (1.9,3.2) {$z_1^{3}$};
\node at (1.9,4.2) {$z_1^{4}$};
\node at (3.8,-3.4) {$c_1$};
\node at (8.5,-3.5) {$c_2$};
\node at (1.2,-2.0) {$P_{1,1}^+$};
\draw [decorate, decoration = {brace, mirror}] (-0.3,-0.1) --  (3.4,-3.3);
\end{scope}

\begin{scope}[xshift=3.5cm, yshift=0cm] 
\coordinate(x1+) at (0.0,0.0);
\coordinate(y11) at (0.0,1.0);
\coordinate(y12) at (0.0,2.0);
\coordinate(y13) at (0.0,3.0);
\coordinate(y14) at (0.0,4.0);
\coordinate(x1-) at (1.5,0.0);
\coordinate(z11) at (1.5,1.0);
\coordinate(z12) at (1.5,2.0);
\coordinate(z13) at (1.5,3.0);
\coordinate(z14) at (1.5,4.0);
\draw (x1+) -- (y14);  
\draw (x1-) -- (z14);  
\draw (x1+) -- (z11);  
\draw (x1-) -- (y11);  
\draw(x1+)[fill=white] circle(\vr);
\draw(x1-)[fill=white] circle(\vr);
\foreach \i in {1,2,3,4} 
{
\draw(y1\i)[fill=white] circle(\vr);
\draw(z1\i)[fill=white] circle(\vr);
}
\node at (-0.4,0.0) {$x_2^{+}$};
\node at (1.9,0.0) {$x_2^{-}$};
\end{scope}

\begin{scope}[xshift=7.0cm, yshift=0cm] 
\coordinate(x1+) at (0.0,0.0);
\coordinate(y11) at (0.0,1.0);
\coordinate(y12) at (0.0,2.0);
\coordinate(y13) at (0.0,3.0);
\coordinate(y14) at (0.0,4.0);
\coordinate(x1-) at (1.5,0.0);
\coordinate(z11) at (1.5,1.0);
\coordinate(z12) at (1.5,2.0);
\coordinate(z13) at (1.5,3.0);
\coordinate(z14) at (1.5,4.0);
\draw (x1+) -- (y14);  
\draw (x1-) -- (z14);  
\draw (x1+) -- (z11);  
\draw (x1-) -- (y11);  
\draw(x1+)[fill=white] circle(\vr);
\draw(x1-)[fill=white] circle(\vr);
\foreach \i in {1,2,3,4} 
{
\draw(y1\i)[fill=white] circle(\vr);
\draw(z1\i)[fill=white] circle(\vr);
}
\node at (-0.4,0.0) {$x_3^{+}$};
\node at (1.9,0.0) {$x_3^{-}$};
\end{scope}

\begin{scope}[xshift=10.5cm, yshift=0cm] 
\coordinate(x1+) at (0.0,0.0);
\coordinate(y11) at (0.0,1.0);
\coordinate(y12) at (0.0,2.0);
\coordinate(y13) at (0.0,3.0);
\coordinate(y14) at (0.0,4.0);
\coordinate(x1-) at (1.5,0.0);
\coordinate(z11) at (1.5,1.0);
\coordinate(z12) at (1.5,2.0);
\coordinate(z13) at (1.5,3.0);
\coordinate(z14) at (1.5,4.0);
\draw (x1+) -- (y14);  
\draw (x1-) -- (z14);  
\draw (x1+) -- (z11);  
\draw (x1-) -- (y11);  
\draw(x1+)[fill=white] circle(\vr);
\draw(x1-)[fill=white] circle(\vr);
\foreach \i in {1,2,3,4} 
{
\draw(y1\i)[fill=white] circle(\vr);
\draw(z1\i)[fill=white] circle(\vr);
}
\node at (-0.4,0.0) {$x_4^{+}$};
\node at (1.9,0.0) {$x_4^{-}$};
\end{scope}
\end{tikzpicture}
\caption{Graph $G_4(F)$ for $F = (x_1 \lor \overline{x}_2 \lor \overline{x}_3) \land (\overline{x}_1 \lor \overline{x}_2 \lor x_4)$}
\label{fig:G4(F)}
\end{center}
\end{figure}

\paragraph{Construction of $H_d(F)$.}  To obtain $H_d(F)$ we modify $G_d(F)$ by adding the edge $x_i^+x_i^-$  for every $i \in [k]$. Further, the paths $P_{i,j}^+$ and $P_{i,j}^-$ of length $d-1$ are removed and replaced, respectively, with paths $R_{i,j}^+$ and $R_{i,j}^-$ of length $d$.
Observe that $H_d(F)$ is a planar graph when $F$ is a Planar 3-SAT (or Planar 1-in-3-SAT) instance. For an example see Fig.~\ref{fig:H4(F)}. 

\begin{figure}[ht!]
\begin{center}
\begin{tikzpicture}[scale=0.9,style=thick,x=1cm,y=1cm]
\def\vr{3pt}
\begin{scope}[xshift=-0cm, yshift=0cm] 
\coordinate(x1+) at (0.0,0.0);
\coordinate(y11) at (0.0,1.0);
\coordinate(y12) at (0.0,2.0);
\coordinate(y13) at (0.0,3.0);
\coordinate(y14) at (0.0,4.0);
\coordinate(x1-) at (1.5,0.0);
\coordinate(z11) at (1.5,1.0);
\coordinate(z12) at (1.5,2.0);
\coordinate(z13) at (1.5,3.0);
\coordinate(z14) at (1.5,4.0);
\coordinate(c1) at (3.5,-3.0);
\coordinate(c2) at (8.5,-3.0);
\coordinate(x11) at (0.875,-0.75);
\coordinate(x12) at (1.75,-1.5);
\coordinate(x13) at (2.625,-2.25);
\coordinate(x21) at (4.625,-0.75);
\coordinate(x22) at (4.25,-1.5);
\coordinate(x23) at (3.875,-2.25);
\coordinate(x31) at (7.25,-0.75);
\coordinate(x32) at (6.0,-1.5);
\coordinate(x33) at (4.75,-2.25);
\coordinate(yy11) at (3.25,-0.75);
\coordinate(yy12) at (5.0,-1.5);
\coordinate(yy13) at (6.75,-2.25);
\coordinate(yy21) at (5.875,-0.75);
\coordinate(yy22) at (6.75,-1.5);
\coordinate(yy23) at (7.625,-2.25);
\coordinate(yy31) at (10.00,-0.75);
\coordinate(yy32) at (9.50,-1.5);
\coordinate(yy33) at (9.00,-2.25);
\draw (x1+) -- (y14);  
\draw (x1+) -- (x1-);
\draw (x1-) -- (z14);  
\draw (x1+) -- (z11);  
\draw (x1-) -- (y11);  
\draw (c1) -- (0,0);  
\draw (c1) -- (5,0);
\draw (c1) -- (8.5,0);
\draw (c2) -- (1.5,0); 
\draw (c2) -- (5,0);
\draw (c2) -- (10.5,0);
\draw(x1+)[fill=white] circle(\vr);
\draw(x1-)[fill=white] circle(\vr);
\draw(c1)[fill=white] circle(\vr);
\draw(c2)[fill=white] circle(\vr);
\draw(x11)[fill=white] circle(\vr);
\draw(x12)[fill=white] circle(\vr);
\draw(x13)[fill=white] circle(\vr);
\draw(x21)[fill=white] circle(\vr);
\draw(x22)[fill=white] circle(\vr);
\draw(x23)[fill=white] circle(\vr);
\draw(x31)[fill=white] circle(\vr);
\draw(x32)[fill=white] circle(\vr);
\draw(x33)[fill=white] circle(\vr);
\draw(yy11)[fill=white] circle(\vr);
\draw(yy12)[fill=white] circle(\vr);
\draw(yy13)[fill=white] circle(\vr);
\draw(yy21)[fill=white] circle(\vr);
\draw(yy22)[fill=white] circle(\vr);
\draw(yy23)[fill=white] circle(\vr);
\draw(yy31)[fill=white] circle(\vr);
\draw(yy32)[fill=white] circle(\vr);
\draw(yy33)[fill=white] circle(\vr);
\foreach \i in {1,2,3,4} 
{
\draw(y1\i)[fill=white] circle(\vr);
\draw(z1\i)[fill=white] circle(\vr);
}
\node at (-0.4,0.3) {$x_1^{+}$};
\node at (1.9,0.3) {$x_1^{-}$};
\node at (-0.4,1.3) {$y_1^{1}$};
\node at (-0.4,2.3) {$y_1^{2}$};
\node at (-0.4,3.3) {$y_1^{3}$};
\node at (-0.4,4.3) {$y_1^{4}$};
\node at (1.9,1.3) {$z_1^{1}$};
\node at (1.9,2.3) {$z_1^{2}$};
\node at (1.9,3.3) {$z_1^{3}$};
\node at (1.9,4.3) {$z_1^{4}$};
\node at (3.8,-3.4) {$c_1$};
\node at (8.5,-3.5) {$c_2$};
\node at (1.2,-2.0) {$R_{1,1}^+$};
\draw [decorate, decoration = {brace, mirror}] (-0.3,-0.1) --  (3.4,-3.3);
\end{scope}

\begin{scope}[xshift=3.5cm, yshift=0cm] 
\coordinate(x1+) at (0.0,0.0);
\coordinate(y11) at (0.0,1.0);
\coordinate(y12) at (0.0,2.0);
\coordinate(y13) at (0.0,3.0);
\coordinate(y14) at (0.0,4.0);
\coordinate(x1-) at (1.5,0.0);
\coordinate(z11) at (1.5,1.0);
\coordinate(z12) at (1.5,2.0);
\coordinate(z13) at (1.5,3.0);
\coordinate(z14) at (1.5,4.0);
\draw (x1+) -- (y14);  
\draw (x1+) -- (x1-);
\draw (x1-) -- (z14);  
\draw (x1+) -- (z11);  
\draw (x1-) -- (y11);  
\draw(x1+)[fill=white] circle(\vr);
\draw(x1-)[fill=white] circle(\vr);
\foreach \i in {1,2,3,4} 
{
\draw(y1\i)[fill=white] circle(\vr);
\draw(z1\i)[fill=white] circle(\vr);
}
\node at (-0.4,0.0) {$x_2^{+}$};
\node at (1.9,0.0) {$x_2^{-}$};
\end{scope}

\begin{scope}[xshift=7.0cm, yshift=0cm] 
\coordinate(x1+) at (0.0,0.0);
\coordinate(y11) at (0.0,1.0);
\coordinate(y12) at (0.0,2.0);
\coordinate(y13) at (0.0,3.0);
\coordinate(y14) at (0.0,4.0);
\coordinate(x1-) at (1.5,0.0);
\coordinate(z11) at (1.5,1.0);
\coordinate(z12) at (1.5,2.0);
\coordinate(z13) at (1.5,3.0);
\coordinate(z14) at (1.5,4.0);
\draw (x1+) -- (y14);  
\draw (x1+) -- (x1-);
\draw (x1-) -- (z14);  
\draw (x1+) -- (z11);  
\draw (x1-) -- (y11);  
\draw(x1+)[fill=white] circle(\vr);
\draw(x1-)[fill=white] circle(\vr);
\foreach \i in {1,2,3,4} 
{
\draw(y1\i)[fill=white] circle(\vr);
\draw(z1\i)[fill=white] circle(\vr);
}
\node at (-0.4,0.0) {$x_3^{+}$};
\node at (1.9,0.0) {$x_3^{-}$};
\end{scope}

\begin{scope}[xshift=10.5cm, yshift=0cm] 
\coordinate(x1+) at (0.0,0.0);
\coordinate(y11) at (0.0,1.0);
\coordinate(y12) at (0.0,2.0);
\coordinate(y13) at (0.0,3.0);
\coordinate(y14) at (0.0,4.0);
\coordinate(x1-) at (1.5,0.0);
\coordinate(z11) at (1.5,1.0);
\coordinate(z12) at (1.5,2.0);
\coordinate(z13) at (1.5,3.0);
\coordinate(z14) at (1.5,4.0);
\draw (x1+) -- (y14);  
\draw (x1+) -- (x1-);
\draw (x1-) -- (z14);  
\draw (x1+) -- (z11);  
\draw (x1-) -- (y11);  
\draw(x1+)[fill=white] circle(\vr);
\draw(x1-)[fill=white] circle(\vr);
\foreach \i in {1,2,3,4} 
{
\draw(y1\i)[fill=white] circle(\vr);
\draw(z1\i)[fill=white] circle(\vr);
}
\node at (-0.4,0.0) {$x_4^{+}$};
\node at (1.9,0.0) {$x_4^{-}$};
\end{scope}
\end{tikzpicture}
\caption{Graph $H_4(F)$ for $F = (x_1 \lor \overline{x}_2 \lor \overline{x}_3) \land (\overline{x}_1 \lor \overline{x}_2 \lor x_4)$}
\label{fig:H4(F)}
\end{center}
\end{figure}

\begin{theorem} \label{thm:bipartite-NP}
    Let $d$ and $p$ be fixed integers with $2 \le d$ and $0 \le p \leq 2d-1$. It is NP-complete to decide whether $\gamma_d^p(G) \leq k$ holds if $G$ is a planar bipartite graph and $k$ is part of the input.
\end{theorem}
\proof
Let $d$ and $p$ be fixed integers complying with the conditions of the theorem. The decision problem clearly belongs to NP. To prove the NP-hardness over the class of bipartite planar graphs, we show a polynomial-time reduction from the Planar 3-SAT problem if $p \leq 2d-3$ and from the Planar 1-in-3-SAT problem if $p \in \{2d-2, 2d-1\}$.
Both problems are known to be NP-complete~\cite{Dyer}.
 Let $F$ be a Planar 3-SAT instance over $k$ variables.  

\begin{description}
    \item[Case 1]:  $p \leq 2d-3$.\\
Assume first that $\gamma_d^p(G_d(F)) \leq k$ and that $D$ is a $\gamma_d^p$-set of $G_d(F)$. For each $i \in [k]$, set $D$ must contain at least one vertex from $X_i$ to $d$-distance dominate $y_i^d$ and $z_i^d$. Therefore, $|D| \leq  k$ implies $|D \cap X_i|=1$, for $i \in [k]$, and $D \subseteq \bigcup_{i \in [k]} X_i$. Let $d_i$ denote the vertex in $D \cap X_i$. Since both $y_i^d$ and $z_i^d$ are $d$-distance dominated by $d_i$, this vertex is either $x_i^+$ or $x_i^-$. Further, every clause vertex $c_j$, for $j \in [\ell]$, must be $d$-distance dominated by a literal vertex. If a clause $C_j$ in the formula $F$ contains a positive literal $x_i$, then $d(c_j,x_i^+)=d-1$ in $G_d(F)$. If $C_j$ does not contain $x_i$, then 
\begin{equation} \label{eq:NP-1}
    d(c_j,x_i^+)\ge \min\{(d-1)+2, 3(d-1)\}= d+1.
\end{equation}
The same is true for negative literals and then, $D$ contains a vertex that represents a literal in $C_j$, for every $j \in [\ell]$. Consequently, the following truth assignment is well-defined and satisfies $F$:
\begin{equation}  \label{eq:NP-2}
    \phi(x_i)= \left\{ \begin{array}{ll}
		\mbox{true}; & x_i^+ \in D, \\		
		\mbox{false}; & x_i^- \in D.
	\end{array} \right. 
\end{equation}
Assume now that $F$ is satisfiable and consider an assignment $\phi: \{x_1, \dots, x_k\} \rightarrow \{ \mbox{true, false}\}$ that satisfies $F$.  Let us set
\begin{equation}  \label{eq:NP-3}
    D=\{x_i^+: \phi(x_i)=\mbox{true, } i \in [k]\} \cup \{x_i^-: \phi(x_i)=\mbox{false, } i \in [k]\}.
\end{equation}
Hence, $|D|=k$. Since $F$ is satisfied by $\phi$, every clause vertex $c_j$ is incident to an $x_i^+,c_j$-path $P_{i,j}^+$ of length $d-1$ such that $x_i^+ \in D$ or to an $x_i^-,c_j$-path $P_{i,j}^-$ of length $d-1$ such that $x_i^- \in D$.  Then every clause vertex is $d$-distance dominated by $D$. Concerning the remaining vertices in $G_d(F)$, we note that if $x_i^+ \in D$, it $d$-distance dominates all vertices from $X_i$, all internal vertices from the existing paths $P_{i,j}^+$ and $P_{i,j}^-$; and the same is true for $x_i^-$. Since $|D\cap \{x_i^+, x_i^-\}|=1$ for each $i\in [k]$, this proves that $D$ is a $d$-distance dominating set. Further, since $d(u,v) \ge 2(d-1)\ge p+1$ holds if $u \in X_i$, $v \in X_{i'}$, and $i \neq i'$, $D$ is a $p$-packing in $G_d(F)$. Therefore, if $p \leq 2d-3$, then $\gamma_d^p(G_d(F)) \leq |D| =k$ if and only if $F$ is satisfiable. 
\item[Case 2]:  $p \in \{2d-2, 2d-1\}$.\\
We show that the formula $F$ is $1$-in-$3$-satisfiable if and only if $\gamma_d^p(G_d(F)) \leq k$. The proof is mostly similar to the argument for Case~1. 

If $D$ is a $\gamma_d^p$-set in $G_d(F)$ and $\gamma_d^p(G_d(F))= |D| \leq k$, then $D$ contains exactly one vertex, namely $x_i^+$ or $x_i^-$ from each $X_i$,  and $D \subseteq \bigcup_{i \in [k]} X_i$ holds. As (\ref{eq:NP-1}) remains true and each clause vertex $c_j$ is $d$-distance dominated, the truth assignment $\phi$ in (\ref{eq:NP-2}) satisfies $F$. Moreover, if $u$ and $v$ are two vertices representing literals from the same clause $C_j$, then $d(u,v)= 2d-2 < p+1$ and hence, $D$ contains at most one of $u$ and $v$. Consequently, each clause $C_j$ is satisfied by exactly one literal in the truth assignment $\phi$. This proves the $1$-in-$3$-satisfiability of $F$.

Assume now that $F$ is $1$-in-$3$-satisfied by $\phi$ and define $D \subseteq V(G_d(F))$ according to (\ref{eq:NP-3}). Then $|D|=k$ and $D$ is a $d$-distance dominating set in $G_d(F)$. Moreover, any two elements of $D$ are literal vertices $u \in X_i$, $v \in X_{i'}$, with $i \neq i'$. Since $F$ is $1$-in-$3$-satisfied by $\phi$, the literals represented by $u$ and $v$ do not occur in a common clause in $F$.
We infer that $d(u,v) \ge \min\{2(d-1)+2, 4(d-1)\} =2d \ge p+1$ holds when $u,v \in D$. Set $D$ is therefore a $d$-distance $p$-packing dominating set in $G_d(F)$ and $\gamma_d^p(G_d(F)) \leq k$. This finishes the proof for Case~$2$.
\end{description}

The polynomial-time reductions established for Cases~$1$ and~$2$ finish the proof.
\qed

\medskip
The next proposition extends Theorem~\ref{thm:bipartite-NP} to $p=2d$, but here the NP-hardness is proved over a wider class of graphs. The question of whether the NP-completeness remains true over the class of planar bipartite graphs remains open. The problem in Proposition~\ref{prop:NP-perfect-code}~$(ii)$ is equivalent to the decision problem of $\gamma_d^{2d} <\infty$ and is known to be NP-complete for every fixed $d\ge 1$ over the class of all graphs \cite{Kratochvil-book}.

\begin{proposition} \label{prop:NP-perfect-code}
For each fixed integer $d \ge 2$, the following problems are NP-complete over the class of planar graphs:
\begin{itemize}
    \item[$(i)$] deciding whether $\gamma_d^{2d}(G) \leq k$ holds where integer $k$ is part of the input;
    \item[$(ii)$] deciding whether $G$ admits a $d$-perfect code.
\end{itemize} 
\end{proposition}
\proof (i) Let $F$ be an instance of the Planar $1$-in-$3$-SAT problem. We are going to prove that $\gamma_d^{2d}(H_d(F)) \leq k$ if and only if $F$ is $1$-in-$3$-satisfiable. Suppose first that $\gamma_d^{2d}(H_d(F)) \leq k$ and $D$ is a $\gamma_d^{2d}$-set in $H_d(F)$. To $d$-distance dominate $y_i^d$ and $z_i^d$, set $D$ contains a vertex from every $X_i$, for $i \in [k]$. As $|D|\leq k$, we infer $|D\cap X_i|=1$ and that the common vertex is a literal vertex. To $d$-distance dominate a clause vertex $c_j$, at least one vertex representing a literal in $C_j$ belongs to $D$. Finally, if $u$ and $v$ represent two literals from the same clause, then $d(u,v)= 2d <2d+1$ and hence, $D$ contains at most one of them. This proves that the following function $1$-in-$3$-satisfies $F$:  
\begin{equation}  \label{eq:NP-4}
    \phi(x_i)= \left\{ \begin{array}{ll}
		\mbox{true}; & x_i^+ \in D, \\		
		\mbox{false}; & x_i^- \in D.
	\end{array} \right. 
\end{equation}

If a truth assignment $\phi$  $1$-in-$3$-satisfies a $3$-SAT instance $F$, we define the set
\begin{equation}  \label{eq:NP-5}
    D=\{x_i^+: \phi(x_i)=\mbox{true, } i \in [k]\} \cup \{x_i^-: \phi(x_i)=\mbox{false, } i \in [k]\},
\end{equation}
and observe that $D$ is a $2d$-packing $d$-distance dominating set in $H_d(F)$ of cardinality $k$. This proves $\gamma_d^{2d}(H_d(F)) = k$.

(ii) We prove that $\gamma_d^{2d}(H_d(F)) < \infty$, that is $H_d(F)$ has a $d$-perfect code, if and only if $F$ is $1$-in-$3$-satisfiable. Suppose that $D'$ is a $d$-distance $2d$-packing dominating set in $H_d(F)$. Since it $d$-distance dominates $y_i^d$ and $z_i^d$, we have $|D'\cap X_i| \ge 1$ for every $i \in [k]$. Further, since the subgraph induced by $X_i$ is of diameter $2d$, no $2d$-packing contains more than one vertex from $X_i$. Therefore, $|D'\cap X_i| = 1$ for every $i \in [k]$, and the common vertex is either $x_i^+$ or $x_i^-$. In either case, no vertex from  $(R_{i,j}^+ \cup R_{i,j}^-) \setminus\{x_i^+, x_i^-\}$ belongs to $D'$ as the distance to the vertex in $D' \cap X_i$ would be at most $d+1 <2d+1$. Consequently, $D'$ consists only of literal vertices.
To $d$-distance dominate a clause vertex $c_j$, $D'$ contains a vertex representing a literal in $C_j$. On the other hand, $D'$ cannot contain two such vertices because their distance is only $2d$. After defining $\phi$ analogously to~(\ref{eq:NP-4}), we may infer that $\phi$ $1$-in-$3$-satisfies $F$.

To establish the other direction, we suppose that a truth assignment $\phi$ $1$-in-$3$-satisfies $F$ and define $D'$ according to~(\ref{eq:NP-5}). It is easy to check that $D'$ is a $2d$-packing
 and also a $d$-distance dominating set in $H_d(F)$.
 \qed
 \medskip

 To conclude the section, we mention that the decision problem of $\gamma_d^{p}(G) \leq k$ can be solved in linear time over the class of trees. It is true for every two nonnegative integers $d$ and $p$ and follows directly from Courcelle's theorem~\cite{Courcelle-1990}.

\section{Cycles}
\label{sec:cycles}

For integers $0 \le p \le d$, every graph has a $d$-distance $p$-packing dominating set. Further, $\gamma_d^p(G)=1$ holds for any graph $G$ with $\rad(G)\leq d $, even if $p> d$. On the other hand, $\gamma_d^p(G)=\infty$ whenever $p>2d$ and $\rad(G) > d$. We have also seen that deciding whether $\gamma_d^{2d}(G) < \infty$ holds is an NP-hard problem for every $d \ge 1$.

For cycles, there are infinitely many examples with $\gamma_d^p(C_n)=\infty$ and $d < p \leq 2d$. For example, $\gamma_2^3(C_n)= \infty$ when $n \in \{6,7,11\}$, and $\gamma_2^4(C_n)=\infty$ when $n \ge 6$ and $n \not\equiv 0\ (\bmod\ {5})$. In this section, we prove the exact value of $\gamma_d^p(C_n)$ for every $p$, $d$, and $n$. First, assume that $p>2d$ and observe that $\gamma_d^p(C_n)=1$ holds when $n \le 2 d+1$ and  $\gamma_d^p(C_n)=\infty$ when $n \ge 2d+2$. Further, for any $d$ and $p$, we have $\gamma_d^p(C_n)= 1 = \lceil\frac{n}{2d+1} \rceil$ when $n \leq 2d+1$. The remaining cases are handled in the following theorem.

\begin{theorem} \label{thm:cycles}
If $d$, $p$, and $n\ge 3$ are integers with  $0 \leq p \leq 2d$ and $n \ge 2d+2$, then 
 $$
    \gamma_d^p(C_n)=
    \begin{cases}
		\left\lceil\frac{n}{2d+1}\right\rceil; & \frac{n}{p+1} \ge \left\lceil\frac{n}{2d+1} \right\rceil,\\
        	\infty; & \mbox{otherwise}.\\
    \end{cases}
    $$    
\end{theorem}
\begin{proof} If $D \subseteq V(C_n)$ and $|D|=k >1$, we denote by $x_1, \dots, x_k$ the distances between the consecutive vertices of $D$ along the cycle (i.e., the lengths of the arcs defined by $D$). Clearly, $\sum_{i=1}^k x_i=n$ and further, $D$ is a $d$-distance $p$-packing dominating set if and only if $p+1 \le x_i \le 2d+1$  for every $i \in [k]$. Therefore, if $D$ is a $\gamma_d^p(C_n)$-set, then $|D|\cdot (2d+1) \ge n$ holds and $\gamma_d^p(C_n) \ge \lceil\frac{n}{2d+1}\rceil$ follows.
\medskip

In the next part, we assume that $\gamma_d^p(C_n) <\infty$ and prove that $\gamma_d^p(C_n) = \lceil\frac{n}{2d+1}\rceil$. Let $D$ be an arbitrary subset of $V(C_n)$ and let $|D|=k \ge 2$. For the sequence $x_1, \dots, x_k$ of the lengths of the arcs defined by $D$, we say that it is \emph{nearly balanced}, if $|x_i -x_j| \leq 1$ for every $1 \leq i <j \leq k$. The set $D$ is nearly balanced if so is the corresponding sequence $x_1, \dots, x_k$.
\begin{claim} \label{claim:cycles-1}
If $D$ is a $d$-distance $p$-packing dominating set in $C_n$, then there exists a nearly balanced $d$-distance $p$-packing dominating set $D''$ in $C_n$ with $|D''|=|D|$.    
\end{claim}
\proof If $D$ is not nearly balanced, then there are two arcs with $x_i +2 \leq x_j$. By replacing $x_i$ with $x_i+1$ and $x_j$ with $x_j-1$, we get a new sequence and a corresponding vertex set $D' \subseteq V(C_n)$. Since $D$ was a $d$-distance $p$-packing dominating set, we have
$$
p+1 \le x_i <x_i +1 \le x_j-1 <x_j \le 2d+1,
$$
which proves that $D'$ is a $d$-distance $p$-packing dominating set in $C_n$. Repeating this procedure if needed, we obtain a nearly balanced $d$-distance $p$-packing dominating set $D''$ with $|D''|=|D|$ at the end. \smallqed

If $|D|= k$ and $D$ is a nearly balanced $d$-distance $p$-packing dominating set, then the lengths of the arcs are $\lfloor\frac{n}{k} \rfloor$ and $\lceil\frac{n}{k} \rceil$, and both lengths occur on some arcs. Consequently, in this case, $p+1 \le \lfloor\frac{n}{k} \rfloor \le \lceil\frac{n}{k} \rceil \le 2d+1 $.
\begin{claim} \label{claim:cycles-2}
If $\gamma_d^p(C_n) <\infty$, then $\gamma_d^p(C_n) =  \left\lceil\frac{n}{2d+1}\right\rceil$.    
\end{claim}
\proof Suppose that $D$ is a $d$-distance $p$-packing dominating set in $C_n$. By Claim~\ref{claim:cycles-1}, we may assume that $D$ is nearly balanced. Let $|D|=k$ and $a= \lceil\frac{n}{2d+1}\rceil$. As we have already proved, $\gamma_d^p(C_n) \ge a$. It implies $k \ge a$. Consider a nearly balanced set $D_a \subseteq V(C_n)$ such that $|D_a|=a$. Note that such a set exists and that the sequence $x_1', \dots , x_a'$ corresponding to $D_a$ consists of entries $\lfloor\frac{n}{a} \rfloor$ and $\lceil\frac{n}{a} \rceil$. Since $D$ is a $d$-distance $p$-packing dominating set and $k \ge a$, we infer
\begin{equation} \label{eq:cycles-1}
   \displaystyle p+1 \le \left\lfloor\frac{n}{k} \right\rfloor \le \left\lfloor\frac{n}{a} \right\rfloor \le \left\lceil\frac{n}{a} \right\rceil \le 2d+1, 
\end{equation}
where the last inequality follows from the fact that $n \leq (2d+1) \lceil\frac{n}{2d+1} \rceil = (2d+1)a $. As (\ref{eq:cycles-1}) implies, $D_a$ is a $d$-distance $p$-packing dominating set and then $\gamma_d^p(C_n) \le \lceil\frac{n}{2d+1} \rceil $. We may now conclude the equality $\gamma_d^p(C_n) = \lceil\frac{n}{2d+1} \rceil $. \smallqed
\medskip

What remains to prove is that $\gamma_d^p(C_n) <\infty$ if and only if $\frac{n}{p+1} \ge \lceil\frac{n}{2d+1} \rceil$. By Claims~\ref{claim:cycles-1} and \ref{claim:cycles-2}, we know that $\gamma_d^p(C_n) <\infty$ holds if and only if the nearly balanced set $D_a$ with $|D_a|=a = \lceil\frac{n}{2d+1} \rceil$ is a $d$-distance $p$-packing dominating set in $C_n$. The sequence corresponding to $D_a$ consists of entries $\lfloor\frac{n}{a} \rfloor$ and $\lceil\frac{n}{a} \rceil$. Recall that $n \leq (2d+1)a $ is always true and implies $\lceil\frac{n}{a}\rceil \leq 2d+1$. Thus, $D_a$ is a $d$-distance $p$-packing dominating set if and only if $p+1 \le \lfloor \frac{n}{a}\rfloor$. Since $p+1$ is an integer, it is equivalent to $p+1 \le \frac{n}{a}$, from which we get $a \le \frac{n}{p+1}$. By substituting $a=\lceil \frac{n}{2d+1}\rceil$, we obtain the desired condition which is equivalent to $\gamma_d^p(C_n) < \infty$.    
\end{proof}
\begin{remark}
\label{rem:Daniel}
  Let $n \ge 2d+2$ and $p <2d$. 
  The property $\gamma_d^p(C_n) <\infty$ is equivalent to the existence of an integer $k$ for which the diophantine equation 
  $$ x_1+x_2 + \cdots + x_k=n
  $$
  has a solution such that $p+1 \le x_i \le 2d+1$ for each $i \in [k]$. This problem can be modeled by considering the numerical semigroup $S=\langle p+1, \dots, 2d+1\rangle$ generated by consecutive integers. Then $\gamma_d^p(C_n) <\infty$ if and only if $n \in S$. For the latter, Garc\' ia-S\' anchez and Rosales~\cite[Corollary 2]{Garcia} proved the necessary and sufficient condition $n\ (\bmod\ p+1) \leq \lfloor \frac{n}{p+1} \rfloor \, (2d-p)$. Since the residue $n\ (\bmod\ p+1)$   equals $n - \lfloor\frac{n}{p+1} \rfloor(p+1),$
  the condition can be rewritten as $\frac{n}{p+1} \ge \lceil\frac{n}{2d+1} \rceil$.
\end{remark}

We close this section by determining $\gamma_d^p(P_n)$ for every path $P_n$. Unlike for cycles, the solution for paths is straightforward. Note first that if $p \ge 2d+1$, then $n \ge 2d+2$ implies $\gamma_d^p(P_n)=\infty$, while $n \le 2d+1$ implies $\gamma_d^p(P_n) = 1$. In the other cases we have the following result. 

\begin{proposition} \label{prop:path}
    For every three integers $d$, $p$, and $n$ with $0 \le p \le 2d$, it holds that $\gamma_d^p(P_n)= \left\lceil\frac{n}{2d+1}\right\rceil$.
\end{proposition}
\begin{proof}
   Since every vertex of the path can $d$-distance dominate at most $2d+1$ vertices, we may infer $\gamma_d^p(P_n)\ge  \lceil\frac{n}{2d+1}\rceil$. Let the vertices of $P_n$ be denoted by the integers $1,2,\dots, n$ in natural order, and write $n$ in the form $n=(2d+1)q -r $ with $0 \le r \le 2d $. Hence, $q= \lceil\frac{n}{2d+1}\rceil$. Let 
   $$D= \{(2d+1)i-d \colon 1 \le i \le q\}$$
   if $r \leq d$, and set 
   $$D= \{(2d+1)i-2d \colon 1 \le i \le q \} $$  if $r > d$. In both cases, $|D|=q$ and $D$ is a $d$-distance $2d$-packing dominating set in $P_n$. This proves $\gamma_d^p(P_n)= \lceil\frac{n}{2d+1}\rceil$ for every $p \leq 2d$.
\end{proof}


\section{Proof of Theorem~\ref{thm:lower-bound-trees-2-domination}}
\label{sec:proof1}

Recall that 
$$\F_2 = \{ T:\ T \text{ tree},\ d(x,y) \equiv 2\ (\bmod\ {5}) \text{ for every } x, y \in S(T), x\ne y\} \setminus \{K_{1,n}:\ n \geq 2\}\,,$$
and that ${\cal T}_2$ is the set of trees in which every two different leaves are at distance $4\ (\bmod\ 5)$. Thus, if $T\in {\cal T}_2$, then every two different support vertices are at distance $2\ (\bmod\ 5)$, that is, ${\cal T}_2 \subseteq \F_2$. For the proof of Theorem~\ref{thm:lower-bound-trees-2-domination} we need the following specific version of~\cite[Lemma 2.1]{Meierling-2005} which is analogous to~\cite[Lemma 2.3]{Meierling-2005}: 

\begin{lemma}
    \label{lem:remove-edge}
    If $T \in \F_2$ has $s$ support vertices, $\ell$ leaves, and $\gamma_2^0(T) \geq 2$, then there exists $uv \in E(T)$ such that $T_u \in \F_2$, $T_v \in \F_2$, $\gamma_2^0(T) = \gamma_2^0(T_u) + \gamma_2^0(T_v)$, $\ell = \ell_u + \ell_v - 2$ and $s = s_u + s_v - 2$, where $\ell_x$ and $s_x$ respectively denote the number of leaves and support vertices in $T_x$\ for $x \in \{u,v\}$.
\end{lemma}

\begin{proof}
    Let $T$ be as in the assumptions of the lemma. Let $P = v_0 v_1 \ldots v_m$ be a diametrical path in $T$ and let $D$ be a $\gamma_2^0(T)$-set. Since $v_0$ can be $2$-distance dominated only by vertices which are also $2$-distance dominated by $v_2$, we may without loss of generality assume that $v_2\in D$. As $\gamma_2^0(T) \geq 2$, we have $\diam(T) \geq 5$, that is, $m \geq 5$. As $T \in \F_2$ and $v_1, v_{m-1}$ are support vertices in $T$, $m-2= d(v_1, v_{m-1}) \equiv 2 \pmod{5}$. Thus $m \equiv 4 \pmod{5}$ and $m \geq 9$. 

    We prove that $\deg(v_i) = 2$ for all $i \in \{3,4,5,6\}$. Suppose to the contrary that for some $i \in \{3,4,5,6\}$ $\deg(v_i) \geq 3$. So there is a leaf $x$ in $T$ such that $d(x,v_i) = d(x,P) \geq 1$. Thus we have
    \begin{align*}
        d(x, v_m) & = d(x, v_i) + d(v_i, v_m) \\
        & = (d(x, v_0) - i) + (m - i)\\
        & = d(x, v_0) + m - 2i.
    \end{align*}
    So $2i = m + d(x, v_0) - d(x, v_m)$ and as $T \in \F_2$ and $x, v_0, v_m$ are leaves in $T$ with different support vertices, we get $2i \equiv 4 + 4 - 4 \equiv 4 \pmod{5}$. But as $i \in \{3,4,5,6\}$, $2i \pmod{5} \in \{1,3,0,2\}$, so we have a contradiction.

    Hence we know that $\deg(v_i) = 2$ for all $i \in \{3,4,5,6\}$ and so we choose $D$ such that $v_7 \in D$. We now prove that $v_4 v_5$ is the edge we are looking for. Let $u = v_4$ and $v=v_5$. As $m \geq 9$, $\ell = \ell_u + \ell_v - 2$ and $s = s_u + s_v - 2$. Clearly, $\gamma_2^0(T_u) = 1$ and $D \setminus \{v_2\}$ is a $\gamma_2^0(T_v)$-set, so $\gamma_2^0(T_u) + \gamma_2^0(T_v) = \gamma_2^0(T)$. It is also easy to see that $T_u \in \F_2$ as $\diam(T_u) = 4$. Let $p,q$ be support vertices in $T_v$. If they are both also support vertices in $T$, then we have $d(p,q) \equiv 2 \pmod{5}$. If $p = v_6$, then $d(v_6, q) = d(v_1, q) - d(v_1, v_6) \equiv 2 - 5 \equiv 2 \pmod{5}$. Thus $T_v \in \F_2$ as well as $\diam(T_v) \geq 4$.
\end{proof}

We are now ready to prove Theorem~\ref{thm:lower-bound-trees-2-domination}. It asserts that if $T$ is a tree on $n$ vertices with $\ell$ leaves and $s$ support vertices, then $$\gamma_2^0(T) \geq \frac{n-\ell-s+4}{5}\,,$$ 
where the equality holds if and only if $T\in \F_2$.

\medskip\noindent 
{\bf Proof of Theorem~\ref{thm:lower-bound-trees-2-domination}}. 
    We prove the bound by induction on $\gamma_2^0(T)$. If $\gamma_2^0(T) = 1$, then $\diam(T) \leq 4$. By considering a central vertex of $T$ we infer that any other vertex must be a leaf or a support vertex, hence $\ell + s \geq n-1$. So $1 \geq \frac{n-\ell-s+4}{5}$. 
    
    From now on, let $T$ be a tree with $\gamma_2^0(T) \geq 2$. By~\cite[Lemma 2.1]{Meierling-2005} there exists $uv \in E(T)$ such that $\gamma_2^0(T) = \gamma_2^0(T_u) + \gamma_2^0(T_v)$. By induction hypothesis, we know that $\gamma_2^0(T_u) \geq \frac{n_u-\ell_u-s_u+4}{5}$ and $\gamma_2^0(T_v) \geq \frac{n_v-\ell_v-s_v+4}{5}$ where $n_x$, $\ell_x$ and $s_x$ denote the number of vertices, leaves and support vertices in $T_x$ for $x \in \{u,v\}$. By definition, $n_u + n_v = n$ and it is not hard to see that $\ell_u + \ell_v \leq \ell + 2$ and $s_u + s_v \leq s + 2$. Thus we have
    \begin{align*}
        \gamma_2^0(T) & = \gamma_2^0(T_u) + \gamma_2^0(T_v)\\
        & \geq \frac{n_u-\ell_u-s_u+4}{5} + \frac{n_v-\ell_v-s_v+4}{5}\\
        & \geq \frac{n - (\ell + 2) - (s+2) + 8}{5}\\
        & = \frac{n-\ell-s+4}{5}. 
    \end{align*}

For the equality part of the theorem, assume that $T$ is a tree with $\gamma_2^0(T) = \frac{n-\ell-s+4}{5}$. We prove that $T \in \F_2$ by induction on $\gamma_2^0(T)$.

    If $\gamma_2^0(T) = 1$, then $\diam(T) \leq 4$ and $n-1=\ell+s$. So there is exactly one vertex $v$ in $T$ that is neither a leaf nor a support. If $T=K_1$, then it belongs to $\F_2$. And if $n \geq 2$, then $N(v)$ is exactly the set of support vertices of $T$ and the remaining vertices are leaves, so $T$ is not a star and $d(x,y) = 2$ for every two support vertices $x,y$ of $T$. Hence, $T \in \F_2$. 

    Let $\gamma_2^0(T) \geq 2$ and assume that all trees with $2$-distance domination number smaller than $\gamma_2^0(T)$ that attain the equality in Theorem~\ref{thm:lower-bound-trees-2-domination} belong to $\F_2$. As $\gamma_2^0(T) = \frac{n-\ell-s+4}{5}$, reconsidering the proof of the inequality,
    we can conclude that $\gamma_2^0(T_x) = \frac{n_x-\ell_x-s_x+4}{5}$ for $x \in \{u,v\}$ and that $\ell_u + \ell_v = \ell + 2$, $s_u + s_v = s + 2$. By the induction hypothesis, this implies $T_u, T_v \in \F_2$. It also follows that $u$ is a leaf in $T_u$ and that its neighbor $u'$ has no other leaves in $T_u$. Similarly, $v$ is a leaf in $T_v$ and its neighbor $v'$ has no other leaves in $T_v$. Let $p,q$ be support vertices of $T$. If $p,q \in V(T_u)$ or $p,q \in V(T_v)$, then $d(p,q) \equiv 2 \pmod{5}$ as $T_u, T_v \in \F_2$. If $p \in V(T_u)$ and $q \in V(T_v)$, then $d(p,q) = d(p,u') + 3 + d(v',q) \equiv 2 + 3 + 2 \equiv 2 \pmod{5}$. Hence, $T \in \F_2$.

    Now let $T \in \F_2$. By induction on $\gamma_2^0(T)$ we prove that $\gamma_2^0(T) = \frac{n-\ell-s+4}{5}$. If $\gamma_2^0(T) = 1$, then $\diam(T) \leq 4$. The only trees with diameter at most 4 that are in $\F_2$ are $K_1$ and trees
    of diameter 4 in which the center is neither support nor leaf. These trees attain equality in Theorem~\ref{thm:lower-bound-trees-2-domination}. 

    Let $\gamma_2^0(T) \geq 2$ and assume that trees in $\F_2$ with $2$-distance domination number smaller than $T$ that are in $\F_2$ attain the equality. By Lemma~\ref{lem:remove-edge} there is an edge $uv \in E(T)$ such that $T_u, T_v \in \F_2$, $\gamma_2^0(T) = \gamma_2^0(T_u) + \gamma_2^0(T_v)$, $\ell = \ell_u + \ell_v - 2$ and $s = s_u + s_v - 2$. By the induction hypothesis, $\gamma_2^0(T_x) = \frac{n_x-\ell_x-s_x+4}{5}$ for $x \in \{u,v\}$. Summing these two equalities and using the expressions for $\ell$ and $s$, we obtain $\gamma_2^0(T) = \frac{n-\ell-s+4}{5}$. 
\qed

\section{Proof of Theorem~\ref{thm:2-2-lower-upper}}    
\label{sec:proof2}

Recall that Theorem~\ref{thm:2-2-lower-upper} asserts that if $T$ is a tree on $n \ge 2$ vertices with $\ell$ leaves and $s$ support vertices, then 
    $$\left \lceil \frac{n-\ell-s+4}{5} \right \rceil \leq \mrho(T) \leq \left \lfloor \frac{n+3s-1}{5} \right \rfloor\,.$$

In the proof we will use the fact that $D \subseteq V(G)$ is a $2$-distance $2$-dominating set if and only if it is a maximal $2$-packing in $G$.

\medskip\noindent 
{\bf Proof of Theorem~\ref{thm:2-2-lower-upper}}. 
The lower bound follows by combining Theorem~\ref{thm:lower-bound-trees-2-domination} with~Proposition~\ref{prop:lower-bound-k-domination}. 

To prove the upper bound, consider a nontrivial tree $T$ with a diametrical path $v_0v_1 \ldots v_{d}$. Then $\diam(T)=d$. We root $T$ in the leaf $v_{d}$. Let $T(v)$ be the vertex set of the subtree rooted in the vertex $v \in V(T)$; that is $T(v)$ contains $v$ and its descendants.  Whenever we define a subtree $T'$, its order and the number of support vertices in $T'$ are denoted by $n'$ and $s'$, respectively, and $D'$ denotes a $\mrho$-set in $T'$. 

If $d \leq 4$, then $\mrho(T)=1$ and the statement holds because we either have $n=s=2$ or $n \ge 3$ and $s \ge 1$. We now assume that $d \ge 5$ and proceed by induction on $n$. 
\begin{description}
    \item [Case 1]: $\deg(v_2) \ge 3$.\\
    Let $T'=T- T(v_1)$ and note that it is not a trivial tree. By our degree condition, $v_2$ is not a leaf in $T'$ and hence, we have $s'=s-1$ and $n'\leq n-2$. For a $\mrho$-set $D'$ in $T'$, it is either a maximal $2$-packing $D$ in $T$ or a maximal $2$-packing in $T$ is obtained as $D=D' \cup \{v_0\}$. Applying the induction hypothesis to $T'$, we obtain
    $$ \mrho(T) \leq |D| \leq |D'|+1 \leq \frac{n'+3s'-1}{5} +1 \leq \frac{n+3s-1}{5}. 
    $$
     \item [Case 2]: $\deg(v_2)=2$ and $\deg(v_3) \ge 3$.\\
      Let $T'=T- T(v_2)$ and observe that $2 \leq n' \leq n -3$. Vertex $v_1$ is the only support vertex of $T$ in $T(v_2)$, and as $\deg(v_3) \ge 3$, no new support vertex arises when we consider $T'$. Hence, $s' =s-1$. The set $D=D' \cup \{v_0\}$ is always a maximal $2$-packing in $T$. Then, we deduce
      $$\mrho(T) \leq |D| = |D'|+1 \leq \frac{n'+3s'-1}{5} +1 \leq \frac{n+3s-2}{5}.
      $$
      \item [Case 3]: $\deg(v_2)=\deg(v_3)=2$ and $\deg(v_4) \ge 3$.\\
      Let $T'=T- T(v_3)$ and observe again that $2 \leq n' \leq n -4$ and $s' =s-1$ hold under the present conditions. The set $D=D' \cup \{v_1\}$ is a maximal $2$-packing in $T$. Therefore, we have
      $$\mrho(T) \leq |D| = |D'|+1 \leq \frac{n'+3s'-1}{5} +1 \leq \frac{n+3s-3}{5}.
      $$
      \item [Case 4]: $\deg(v_2)=\deg(v_3)=\deg(v_4)= 2$.\\
      If $\deg(v_5)=1$, then $v_5$ is the root, $n \ge 6$, and $s=2$. For this case, we can easily determine $\mrho(T)=2 <\frac{n+3s-1}{5}$. If $\deg(v_5) \ge 2$, consider the nontrivial tree $T'=T- T(v_4)$. Its order is $n' \le n-5$. The number of support vertices in $T'$ is either $s-1$ or $s$ depending on whether $\deg(v_5)\ge 3$ or  $\deg(v_5)=2$. For a $\mrho$-set $D'$ of $T'$, the superset $D=D' \cup \{v_2\}$ is always a maximal $2$-packing in $T$. We then conclude
      $$\mrho(T) \leq |D| = |D'|+1 \leq \frac{n'+3s'-1}{5} +1 \leq \frac{n+3s-1}{5}.
      $$
      \end{description}
      As the above cases cover all possibilities and $\mrho(T)$ is an integer, the upper bound follows.
\qed      

\medskip
The upper bound in Theorem~\ref{thm:2-2-lower-upper} is sharp for all nontrivial paths except those of order $5k$. Further sharp examples are obtained by taking graphs $G_p-e_1$ as defined in Section~\ref{sec:spanning-trees} for $p \in \{2,3,4\}$.

As already mentioned, in~\cite{Henning-1998-1} Henning proved that $\mrho(T) \leq \frac{n+3-2\sqrt{n}}{2}$. Since $\frac{n+3s-1}{5} < \frac{n+3-2\sqrt{n}}{2}$ if and only if 
$$s < \frac{n}{2} - \frac{5}{3}\sqrt{n} + \frac{17}{6}\,,$$
the bounds are independent and, moreover, the one of Theorem~\ref{thm:2-2-lower-upper} is sharper in most cases.

\section{Proof of Theorem~\ref{thm:trees-upper-bound-gamma_d^2}}
\label{sec:proof3}

For this proof, we recall the following auxiliary result that will be useful to us. 
\begin{lemma} {\rm (Henning, Oellermann and Swart~\cite{Henning1995})}
    \label{lemma:d-distance-dominating-set}
If $G$ is a connected graph with $\rad(G) \ge d \ge 1$, then there exists a smallest $d$-distance dominating set $D$ of $G$ such that for each $v \in D$, there exists a vertex $w \in V(G) \setminus D$ at distance exactly $d$ from $v$ such that $N_d[w] \cap D=\{v\}$.
\end{lemma}

Recall that Theorem~\ref{thm:trees-upper-bound-gamma_d^2} asserts that if $d \geq 2$ and $T$ is a tree of order $n$, then $\gamma_d^2(T) \leq \frac{n-2\sqrt{n}+d+1}{d}$.

\medskip\noindent 
{\bf Proof of Theorem~\ref{thm:trees-upper-bound-gamma_d^2}}. 
    If $T$ is a tree with $\rad(T)\le d$, then $\gamma_d^2(T)=1 =\frac{1-2\sqrt{1}+d+1}{d} \leq \frac{n-2\sqrt{n}+d+1}{d}$. Hence assume in the rest that $\rad(T)\ge d+1$. Let $D=\{v_1,v_2,\dots,v_b\}$ be a $\gamma_d^0(T)$-set of $T$ satisfying the statement of Lemma~\ref{lemma:d-distance-dominating-set}. For each $i \in [b]$, let
    \begin{align*}
    W_i & =  \{w \in V(T) \setminus D:\ d(v_i,w)=d \text{~and~} N_d[w] \cap D =\{v_i\}\}, \\
    X_i & = \{x \in V(T):\ x \text{~belongs to a~} v_i,w\text{-path of order~} d+1 \text{~for some~} w \in W_i\}, \\
    U_i &  = \{u \in X_i:\ u \in N(v_i)\}, \\
    Z_i &  = \{z \in X_i:\ z \in N_2[v_i]\setminus N[v_i]\}.
\end{align*}	
    By Lemma~\ref{lemma:d-distance-dominating-set}, $W_i,U_i,Z_i \neq \emptyset$. Moreover, $|X_i| \geq d+1$.
	
    \begin{claim}
        \label{claim:disjoint}
	$X_i \cap X_j =\emptyset$ for $1 \leq i < j \leq b$.
    \end{claim}
	
    \proof
         Suppose that there exists a vertex $x \in X_i \cap X_j$ for some $i$ and $j$. Then there exist two vertices $w_i \in W_i$ and $w_j \in W_j$ such that the $v_i,w_i$-path and the $v_j,w_j$-path contain the vertex $x$. It is easy to see that $\{v_i,v_j\} \subseteq N_d[w_i] \cap D$ or $\{v_i,v_j\} \subseteq N_d[w_j] \cap D$, a contradiction.
    \smallqed
\medskip

    Since $D$ $d$-distance dominates $V(T)$, and by Claim \ref{claim:disjoint}, we can partition $V(T)$ into $V_1,\dots,V_b$, where each set $V_i$ induces a tree $T_i$ of radius at most $d$. For each $i \in [b]$, $X_i \subseteq V_i$ and $v_i$ $d$-distance dominates $V_i$. Let $n_i=|V_i|$. Then $n_i \geq d+1$. By the pigeonhole principle, we may assume that $n_1 \geq \frac{n}{b}$. For each $i \in [b]$, we further let 
    \begin{align*}
        W_i' & =  \{w \in V_i:\ d(v_i,w)=d\}, \\
        X_i' & = \{x \in V_i:\ x \text{~belongs to a~} v_i,w\text{-path of order~} d+1 \text{~for some~} w \in W_i'\}, \\
        U_i' &  = \{u \in X_i':\ u \in N(v_i)\}, \\
        Z_i' &  = \{z \in X_i':\ z \in N_2[v_i]\setminus N[v_i]\}.
    \end{align*}

    For each $1 \leq i < j \leq b$, $T_i$ joins $T_j$ by at most one edge in $T$. We can regard each $T_i$ as a vertex $y_i$, and then we get a tree $T'$ of order $b$. Then $T_i$ joins $T_j$ if and only if $y_iy_j \in E(T')$. We also relabel the vertices such that for every $i\in [b]$ the vertices $y_1, \ldots, y_i$ induce a tree.
    
    Consider the tree $T_i$ and root it at $v_i$. For $j \in \{0\} \cup [d]$, let $L_j^i$ be the set of vertices of the $j$-th level of $T_i$. Let $U_i'=\{u_1^i,u_2^i,\dots,u_s^i\}$ where $s \geq 1$. For each $t \in [s]$, let $T_{u_t^i}$ be the tree induced from $T_i$ by the descendants of $u_t^i$ and $\{u_t^i,v_i\}$. Let $z_t^i \in V(T_{u_t^i}) \cap Z_i'$ be a vertex.
    
    Now we will construct a $d$-distance $2$-packing dominating set $S$ of $T$ in $b$ steps. First, let $S_1=\{v_1\}$. For the $i$-th step, suppose that we have got the set $S_{i-1}$ by the following procedure. Let $v \in V(T_{i'})$ for some $i' < i$ and $v$ is joined to some vertex $v'$ of $L_{j}^{i}$. Clearly, there is exactly one such vertex $v$ for $T_i$ and $d(v,v')=1$. Let $d(v',S_{i-1})$ be the shortest distance between $v'$ and the vertices of $S_{i-1}$ in $T$. We build $S_i$ form $S_{i-1}$ as follows:

    \begin{enumerate}[(1)]
    \item If $j=0$, i.e.\ $v'=v_i$, then let
    $$
    S_i'=
    \begin{cases}
	\bigcup_{t \in [s]}\{z_{t}^{i}\}; & d(v_i,S_{i-1})=1 ,\\
	\{u_1^i\} \cup \bigcup_{t \in [s] \setminus \{1\}}\{z_{t}^{i}\}; & d(v_i,S_{i-1})=2,\\
	\{v_i\}; & d(v_i,S_{i-1}) \geq 3.\\
    \end{cases}
    $$ 

    \item If $j=1$, then let
    $$
    S_i'=
    \begin{cases}
	\{u_1^i\} \cup \bigcup_{t \in [s] \setminus \{1\}}\{z_{t}^{i}\}; & d(v',S_{i-1})=1 \text{ and } v' \in N_{T_i}(v_i) \setminus U_i, \\
	\{u_2^i\} \cup \bigcup_{t \in [s] \setminus \{1,2\}}\{z_{t}^{i}\}; & d(v',S_{i-1})=1 \text{ and } v' = u_1^i \footnotemark,\\
	\{v_i\}; & d(v',S_{i-1}) \geq 2.\\
    \end{cases}
    $$ 
    \footnotetext{If $v' \in U_i$, we may without loss of generality assume that $v' = u_1^i$. Note that $S_i'= \emptyset$ is possible in this case.}

    \item If $j \geq 2$, then let $S_{i}'= \{v_i\}$.

    \item For each $j \in \{0\} \cup [d]$, let $S_{i}=S_{i-1} \cup S_i'$.
    \end{enumerate}

    By the procedure, we know that $S=S_b$ is a $d$-distance $2$-packing dominating set, and $|S_i'| \leq \frac{n_i-1}{d}$ for $i \in [b] \setminus \{1\}$. Since $S=\{v_1\} \cup \bigcup_{i \in [b] \setminus \{1\}}S_i'$, $n_1 \geq \frac{n}{b}$, and $b+\frac{n}{b} \geq \sqrt{n}+\frac{n}{\sqrt{n}}=2\sqrt{n}$, we have
    \begin{equation*}
	\begin{aligned}
		\gamma_d^2(T) \leq |S| &\leq 1+\sum_{i \in [b] \setminus \{1\}}\frac{n_i-1}{d}\\
		&\leq 1+\frac{1}{d}(n-\frac{n}{b}-b+1)\\
		&\leq \frac{n-2\sqrt{n}+d+1}{d}.
	\end{aligned}
    \end{equation*}
\qed    

\medskip
The upper bound in~Theorem \ref{thm:trees-upper-bound-gamma_d^2} is best possible as demonstrated by the following example. Let $T$ be the tree obtained from a star $K_{1,ds}~(s \geq 0)$ by attaching $s$ copies of $P_d$ to each vertex of the star. One can see Fig.~\ref{fig:sharp-example-upper-bound-gamma_d^2} for an illustration for this construction.

    \begin{figure}[ht!]
	\begin{center}
		\begin{tikzpicture}[scale=.36]
			\tikzstyle{vertex}=[circle, draw, inner sep=0pt, minimum size=6pt]
			\tikzset{vertexStyle/.append style={rectangle}}
			
			\vertex (1) at (0,0) [scale=1.00,fill=white] {};
			
                \vertex (2) at (-4,0) [scale=1.00,fill=white] {};

                \vertex (3) at (4,0) [scale=1.00,fill=white] {};

                \vertex (4) at (-4,4) [scale=1.00,fill=white] {};

                \vertex (5) at (-4,-4) [scale=1.00,fill=white] {};

                \vertex (6) at (4,4) [scale=1.00,fill=white] {};

                \vertex (7) at (4,-4) [scale=1.00,fill=white] {};

                \vertex (8) at (-6,5) [scale=1.00,fill=white] {};

                \vertex (9) at (-6,3) [scale=1.00,fill=white] {};

                \vertex (10) at (-8,5) [scale=1.00,fill=white] {};

                \vertex (11) at (-8,3) [scale=1.00,fill=white] {};

                \vertex (12) at (-10,5) [scale=1.00,fill=white] {};

                \vertex (13) at (-10,3) [scale=1.00,fill=white] {};

                \vertex (14) at (-6,1) [scale=1.00,fill=white] {};

                \vertex (15) at (-6,-1) [scale=1.00,fill=white] {};
                
                \vertex (16) at (-8,1) [scale=1.00,fill=white] {};

                \vertex (17) at (-8,-1) [scale=1.00,fill=white] {};

                \vertex (18) at (-10,1) [scale=1.00,fill=white] {};

                \vertex (19) at (-10,-1) [scale=1.00,fill=white] {};

                \vertex (20) at (-6,-3) [scale=1.00,fill=white] {};

                \vertex (21) at (-6,-5) [scale=1.00,fill=white] {};
                
                \vertex (22) at (-8,-3) [scale=1.00,fill=white] {};

                \vertex (23) at (-8,-5) [scale=1.00,fill=white] {};

                \vertex (24) at (-10,-3) [scale=1.00,fill=white] {};

                \vertex (25) at (-10,-5) [scale=1.00,fill=white] {};

                \vertex (26) at (6,-3) [scale=1.00,fill=white] {};

                \vertex (27) at (6,-5) [scale=1.00,fill=white] {};
                
                \vertex (28) at (8,-3) [scale=1.00,fill=white] {};

                \vertex (29) at (8,-5) [scale=1.00,fill=white] {};

                \vertex (30) at (10,-3) [scale=1.00,fill=white] {};

                \vertex (31) at (10,-5) [scale=1.00,fill=white] {};

                \vertex (32) at (6,1) [scale=1.00,fill=white] {};

                \vertex (33) at (6,-1) [scale=1.00,fill=white] {};
                
                \vertex (34) at (8,1) [scale=1.00,fill=white] {};

                \vertex (35) at (8,-1) [scale=1.00,fill=white] {};

                \vertex (36) at (10,1) [scale=1.00,fill=white] {};

                \vertex (37) at (10,-1) [scale=1.00,fill=white] {};

                \vertex (38) at (6,5) [scale=1.00,fill=white] {};

                \vertex (39) at (6,3) [scale=1.00,fill=white] {};

                \vertex (40) at (8,5) [scale=1.00,fill=white] {};

                \vertex (41) at (8,3) [scale=1.00,fill=white] {};

                \vertex (42) at (10,5) [scale=1.00,fill=white] {};

                \vertex (43) at (10,3) [scale=1.00,fill=white] {};

                \vertex (44) at (-1,-3) [scale=1.00,fill=white] {};

                \vertex (45) at (1,-3) [scale=1.00,fill=white] {};

                \vertex (46) at (-1,-5) [scale=1.00,fill=white] {};

                \vertex (47) at (1,-5) [scale=1.00,fill=white] {};

                \vertex (48) at (-1,-7) [scale=1.00,fill=white] {};

                \vertex (49) at (1,-7) [scale=1.00,fill=white] {};
            
			\path
			(1) edge (2)
                (1) edge (3)
                (1) edge (4)
                (1) edge (5)
                (1) edge (6)
                (1) edge (7)
                (4) edge (8)
                (4) edge (9)
                (8) edge (10)
                (10) edge (12)
                (9) edge (11)
                (11) edge (13)
                (2) edge (14)
                (2) edge (15)
                (14) edge (16)
                (16) edge (18)
                (15) edge (17)
                (17) edge (19)
                (5) edge (20)
                (5) edge (21)
                (20) edge (22)
                (22) edge (24)
                (21) edge (23)
                (23) edge (25)
                (7) edge (26)
                (7) edge (27)
                (26) edge (28)
                (28) edge (30)
                (27) edge (29)
                (29) edge (31)
                (3) edge (32)
                (3) edge (33)
                (32) edge (34)
                (34) edge (36)
                (33) edge (35)
                (35) edge (37)
                (6) edge (38)
                (6) edge (39)
                (38) edge (40)
                (40) edge (42)
                (39) edge (41)
                (41) edge (43)
                (1) edge (44)
                (1) edge (45)
                (44) edge (46)
                (46) edge (48)
                (45) edge (47)
                (47) edge (49)
			;
			
		\end{tikzpicture}
		\caption{An example where $s=2$ and $d=3$.}
		\label{fig:sharp-example-upper-bound-gamma_d^2}
	\end{center}
\end{figure}
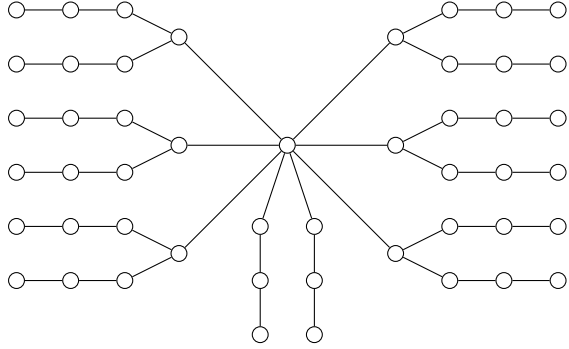

Note that
\begin{equation*}
    \begin{gathered}
	n=|V(T)|=ds+1+ds(ds+1)=(ds+1)^2,\\
		\gamma_d^2(T)=1+ds^2,
    \end{gathered}
\end{equation*}
and
\begin{equation*}
    \gamma_d^2(T)=1+ds^2=\frac{(ds+1)^2-2(ds+1)+d+1}{d}=\frac{n-2\sqrt{n}+d+1}{d}.
\end{equation*}

Inspired by the result from~\cite{Gimbel-1996} asserting that $\gamma_d^1(G) \leq \frac{n-2\sqrt{n}+d+1}{d}$, and by Theorem \ref{thm:trees-upper-bound-gamma_d^2}, we propose the following conjecture.

\begin{conjecture}
    \label{conj:trees-upper-bound-gamma_d^p}
If $3 \leq p \leq d$ and $T$ is a tree of order $n$, then $\gamma_d^p(T) \leq \frac{n-2\sqrt{n}+d+1}{d}$.
\end{conjecture}

Note that if Conjecture~\ref{conj:trees-upper-bound-gamma_d^p} is true, then the above example can be used to show that the bound is sharp.

\section{Spanning trees}
\label{sec:spanning-trees}

Let $e$ be an edge of a graph $G$. Since $d_{G-e}(u,v) \ge d(u,v)$ for every $u,v\in V(G)$, a packing of $G$ remains a packing in $G-e$. On the other hand, this property in general does not extend to $\mrho(G)$. To see it, consider the graphs $G_{p}$ and $G_{p}'$ ($p\ge 2$) as shown in Fig.~\ref{fig:G1&G2}, where $\deg(u_i) = \deg(v_i) = p+2$ for $i \in [2]$.

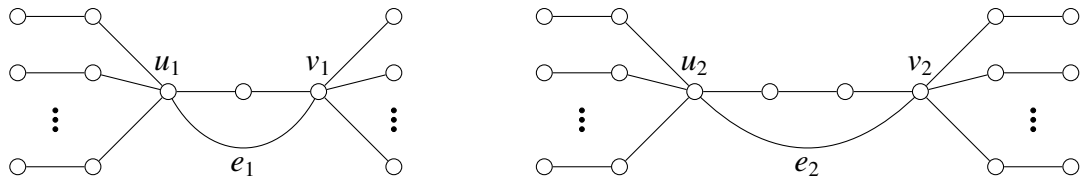
\begin{figure}[ht!]
	\begin{center}
		\begin{tikzpicture}[scale=.5]
			\tikzstyle{vertex}=[circle, draw, inner sep=0pt, minimum size=6pt]
			\tikzset{vertexStyle/.append style={rectangle}}

            \draw (-4,0)..controls (-2,-2) and (0,-2) ..(2,0);
            
			\vertex (1) at (-4,0) [scale=1.00,fill=white] {};
			\node ($u_2$) at (-4,0.7) {$u_2$};
			
			\vertex (2) at (-6,2) [scale=1.00,fill=white] {};
			
			\vertex (21) at (-8,2) [scale=1.00,fill=white] {};
			
			\vertex (3) at (-6,0.5) [scale=1.00,fill=white] {};
			
			\vertex (31) at (-8,0.5) [scale=1.00,fill=white] {};
			
			\vertex (3-1) at (-7,-0.5) [scale=0.3,fill=black] {};
			
			\vertex (3-2) at (-7,-0.75) [scale=0.3,fill=black] {};
			
			\vertex (3-3) at (-7,-1) [scale=0.3,fill=black] {};
			
			\vertex (4) at (-6,-2) [scale=1.00,fill=white] {};
			
			\vertex (41) at (-8,-2) [scale=1.00,fill=white] {};
			
			\vertex (5) at (-2,0) [scale=1.00,fill=white] {};
			
			\vertex (6) at (0,0) [scale=1.00,fill=white] {};
			
			\vertex (7) at (2,0) [scale=1.00,fill=white] {};
			\node ($v_2$) at (2,0.7) {$v_2$};
			
			\vertex (8) at (4,2) [scale=1.00,fill=white] {};
			
			\vertex (81) at (6,2) [scale=1.00,fill=white] {};
			
			\vertex (8-1) at (5,-0.5) [scale=0.3,fill=black] {};
			
			\vertex (8-2) at (5,-0.75) [scale=0.3,fill=black] {};
			
			\vertex (8-3) at (5,-1) [scale=0.3,fill=black] {};
			
			\vertex (9) at (4,0.5) [scale=1.00,fill=white] {};
			
			\vertex (91) at (6,0.5) [scale=1.00,fill=white] {};
			
			\vertex (10) at (4,-2) [scale=1.00,fill=white] {};
			
			\vertex (101) at (6,-2) [scale=1.0,fill=white] {};
			
			\node ($e_2$) at (-1,-2) {$e_2$};
			
			\path
			(1) edge (2)
			(1) edge (3)
			(1) edge (4)
			(1) edge (5)
			(5) edge (6)
			(6) edge (7)
			(7) edge (8)
			(7) edge (9)
			(7) edge (10)
			(2) edge (21)
			(3) edge (31)
			(4) edge (41)
		    (8) edge (81)
		    (9) edge (91)
		    (10) edge (101)
			;

			\draw (-18,0)..controls (-17,-2) and (-15,-2) ..(-14,0);
			
			\vertex (11) at (-18,0) [scale=1.00,fill=white] {};
			\node ($u_1$) at (-18,0.7) {$u_1$};
			
			\vertex (12) at (-20,2) [scale=1.00,fill=white] {};
			
			\vertex (121) at (-22,2) [scale=1.00,fill=white] {};
			
			\vertex (13) at (-20,0.5) [scale=1.00,fill=white] {};
			
			\vertex (131) at (-22,0.5) [scale=1.00,fill=white] {};
			
			\vertex (13-1) at (-21,-0.5) [scale=0.3,fill=black] {};
			
			\vertex (13-2) at (-21,-0.75) [scale=0.3,fill=black] {};
			
			\vertex (13-3) at (-21,-1) [scale=0.3,fill=black] {};
			
			\vertex (14) at (-20,-2) [scale=1.00,fill=white] {};
			
			\vertex (141) at (-22,-2) [scale=1.00,fill=white] {};
			
			\vertex (15) at (-16,0) [scale=1.00,fill=white] {};
			
			\vertex (16) at (-14,0) [scale=1.00,fill=white] {};
			\node ($v_1$) at (-14,0.7) {$v_1$};
			
			\vertex (18) at (-12,2) [scale=1.00,fill=white] {};
			
			\vertex (18-1) at (-12,-0.5) [scale=0.3,fill=black] {};
			
			\vertex (18-2) at (-12,-0.75) [scale=0.3,fill=black] {};
			
			\vertex (18-3) at (-12,-1) [scale=0.3,fill=black] {};
			
			\vertex (19) at (-12,0.5) [scale=1.00,fill=white] {};
			
			\vertex (20) at (-12,-2) [scale=1.00,fill=white] {};
			
			\node ($e_1$) at (-16,-2) {$e_1$};
			
			\path
			(11) edge (12)
			(11) edge (13)
			(11) edge (14)
			(11) edge (15)
			(15) edge (16)
			(16) edge (18)
			(16) edge (19)
			(16) edge (20)
			(12) edge (121)
			(13) edge (131)
			(14) edge (141)
			;

		\end{tikzpicture}
		\caption{The graphs $G_p$ (left) and $G_p'$ (right).}
		\label{fig:G1&G2}
	\end{center}
\end{figure}

Note that $\ecc_{G_p}(u_1) = 2$, hence $\mrho(G_p)=1$. Since $\mrho(G_p-e_1)=1+p$, this example shows that $\mrho$ can increase arbitrary by removing an edge. On the other hand, $\{u_2, v_2\}$ is a $\mrho$-set of $G_p'-e$, so $\mrho(G_p'-e_2)=2$, but $\mrho(G_p') =  1+p$, hence $\mrho$ can also decrease arbitrary by removing an edge.

The sequential removal of edges can therefore give rise to a non-monotonic sequence of $\mrho$ of associated graphs. Nevertheless, we have the following monotonicity result. 

\begin{proposition}
\label{prop:spanning-trees}
Every connected graph $G$ contains a spanning tree $T$ such that $\mrho(T) \leq \mrho(G)$.
\end{proposition}

\begin{proof}
Let $S = \{s_1, \ldots, s_{\mrho(G)}\}$ be a $\mrho$-set of $G$. Let $S' = \{v\in V(G)\setminus S:\ N(v)\cap S \ne \emptyset\}$, and let $S'' = \{v\in V(G)\setminus S:\ N(v)\cap S = \emptyset\}$. Then $V(G) = S \cup S'\cup S''$.

We are going to construct a required spanning tree $T$ as follows. First add all the vertices of $S$ to $T$. Consider next the vertices from $S'$. Then, by definition, each $v\in S'$ is adjacent to some vertex $s_i\in S$, and, moreover, since $S$ is a $2$-packing, $N(v) \cap S = \{s_i\}$. We now add the edge $vs_i$ to $T$. After this is done for all the vertices of $S'$, the so far constructed $T$ is a disjoint union of stars with the vertices of $S$ being the centers of the stars. Consider finally the vertices from $S''$. If $v\in S''$, then since $S$ is a maximal $2$-packing, there exists a vertex $s_j$ such that $d(v,s_j) = 2$. There could be more vertices in $S$ at distance two from $v$, but we select and fix $s_j$. Let $v, u, s_j$ be a shortest $v,s_j$-path and note that $u$ has already been added to $T$ while considering the set $S'$. We now add the edge $vu$ to $T$. We do this procedure for each of the vertices of $S''$ one by one.

After the above described procedure is finished, the so far constructed $T$ forms a spanning forest of $G$ consisting of $\mrho(G)$ components (where each component contains exactly one vertex from $S$). We complete the construction of $T$ by adding $\mrho(G)-1$ additional edges of $G$ to obtain the final spanning tree $T$. Note that in this way, having in mind that $S$ is a $2$-packing in $G$, the set $S$ is also a $2$-packing in $T$. Moreover, by the construction, for each vertex $v\in V(T)$ there exists a vertex $s_i\in S$ such that $d_T(v,s_i) \le 2$. We can conclude that $S$ is a maximal $2$-packing of $T$ and thus $\mrho(T) \le |S| = \mrho(G)$.
\end{proof}

Related to Proposition~\ref{prop:spanning-trees}, we recall that every connected graph contains a spanning tree with the same domination number, see~\cite[Exercise 10.14]{Imrich-2008}, and that the same holds for the total domination number, see~\cite[Lemma 4.41]{Haynes-2023} and, more generally, for the $d$-distance domination number $\gamma_d^0(G)$ for every $d \ge 1$, see~\cite[Lemma 2.1]{Davila-2017}. To see that we cannot derive this conclusion for $\mrho$ consider the graph $H_p$, $p\ge 2$, as shown in Fig.~\ref{fig:spanning-tree}, where $\deg(v_i)=p+2$ for $i \in [3]$.

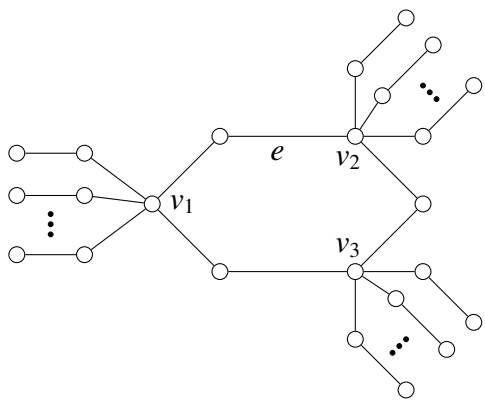
\begin{figure}[ht!]
	\begin{center}
		\begin{tikzpicture}[scale=.45]
			\tikzstyle{vertex}=[circle, draw, inner sep=0pt, minimum size=6pt]
			\tikzset{vertexStyle/.append style={rectangle}}
			
			\vertex (1) at (-4,0) [scale=1.00,fill=white] {};
			
			\vertex (2) at (-2,2) [scale=1.00,fill=white] {};
			
			\vertex (3) at (-2,-2) [scale=1.00,fill=white] {};
			
			\vertex (4) at (2,2) [scale=1.00,fill=white] {};
			
			\vertex (5) at (2,-2) [scale=1.00,fill=white] {};
			
			\vertex (6) at (4,0) [scale=1.00,fill=white] {};
			
			\vertex (7) at (-6,1.5) [scale=1.00,fill=white] {};
			
			\vertex (8) at (-6,0.25) [scale=1.00,fill=white] {};
			
			\vertex (9) at (-6,-1.5) [scale=1.00,fill=white] {};
			
			\vertex (10) at (-8,1.5) [scale=1.00,fill=white] {};
			
			\vertex (11) at (-8,0.25) [scale=1.00,fill=white] {};
			
			\vertex (12) at (-8,-1.5) [scale=1.00,fill=white] {};
			
			\vertex (10-1) at (-7,-0.3) [scale=0.3,fill=black] {};
			
			\vertex (11-1) at (-7,-0.6) [scale=0.3,fill=black] {};
			
			\vertex (12-1) at (-7,-0.9) [scale=0.3,fill=black] {};

			\vertex (13) at (2,4) [scale=1.00,fill=white] {};

            \vertex (14) at (2.8,3.2) [scale=1.00,fill=white] {};

            \vertex (15) at (4,2) [scale=1.00,fill=white] {};

            \vertex (16) at (3.5,5.5) [scale=1.00,fill=white] {};

            \vertex (17) at (4.3,4.7) [scale=1.00,fill=white] {};

            \vertex (18) at (5.5,3.5) [scale=1.00,fill=white] {};
            
            \vertex (16-1) at (4,3.5) [scale=0.3,fill=black] {};
            
            \vertex (17-1) at (4.2,3.3) [scale=0.3,fill=black] {};
            
            \vertex (18-1) at (4.4,3.1) [scale=0.3,fill=black] {};

            \vertex (19) at (2,-4) [scale=1.00,fill=white] {};
            
            \vertex (20) at (3.2,-2.8) [scale=1.00,fill=white] {};
            
            \vertex (21) at (4,-2) [scale=1.00,fill=white] {};
            
            \vertex (22) at (3.5,-5.5) [scale=1.00,fill=white] {};
            
            \vertex (23) at (4.7,-4.3) [scale=1.00,fill=white] {};
            
            \vertex (24) at (5.5,-3.5) [scale=1.00,fill=white] {};
            
            \vertex (22-1) at (3.5,-4) [scale=0.3,fill=black] {};
            
            \vertex (23-1) at (3.3,-4.2) [scale=0.3,fill=black] {};
            
            \vertex (24-1) at (3.1,-4.4) [scale=0.3,fill=black] {};	

            \node ($e$) at (-0.3,1.5) {$e$};

            \node ($v_1$) at (-3.1,0) {$v_1$};

            \node ($v_2$) at (1.8,1.3) {$v_2$};

            \node ($v_3$) at (1.8,-1.3) {$v_3$};
            
			\path
			(1) edge (2)
			(1) edge (3)
			(2) edge (4)
			(3) edge (5)
			(6) edge (4)
			(6) edge (5)
			(1) edge (7)
			(1) edge (8)
			(1) edge (9)
			(7) edge (10)
			(8) edge (11)
			(9) edge (12)
			(4) edge (13)
			(4) edge (14)
			(4) edge (15)
			(13) edge (16)
			(14) edge (17)
			(15) edge (18)
			(5) edge (19)
			(5) edge (20)
			(5) edge (21)
			(19) edge (22)
			(20) edge (23)
			(21) edge (24)
			;
			
		\end{tikzpicture}
		\caption{The graph $H_p$ with $\mrho(T) < \mrho(H_p)$ for each spanning tree $T$ of $H_p$.}
		\label{fig:spanning-tree}
	\end{center}
\end{figure}

If a maximal packing $S$ of $H_p$ contains a vertex $v_i$, then $S\cap\{v_1, v_2, v_3 \} = \{v_i\}$. From this fact we can deduce that $\mrho(H_p) = 1 + 2p$. The graph $H_p$ contains exactly six spanning trees, by the symmetry we may consider the spanning tree $H_p-e$, where $e$ is the edge as in Fig.~\ref{fig:spanning-tree}. Then $d_{H_p-e}(v_1, v_2) = 4$ and the set containing $v_1$, $v_2$, and the leaves that are at distance two from $v_3$, is a $\mrho$-set of $H_p-e$, so that $\mrho(G-e)=p+2$. Hence, since $p\ge 2$, we can conclude that $\mrho(T) < \mrho(H_p)$ for every spanning tree $T$ of $H_p$. 

\section*{Acknowledgements}

We would like to express our sincere gratitude to Daniel Smertnig for discussions on the case of cycles. In particular, he  directed us to the article~\cite{Garcia} leading in turn to Remark~\ref{rem:Daniel}.
Csilla Bujt\'{a}s, Vesna Ir\v{s}i\v{c}, and Sandi Klav\v{z}ar were supported by the Slovenian Research and Innovation Agency (ARIS) under the grants P1-0297, N1-0285, N1-0355, and  Z1-50003.  Vesna Ir\v{s}i\v{c} also acknowledges the financial support from the European Union (ERC, KARST, 101071836).

\section*{Declaration of interests}
 
The authors declare that they have no conflict of interest. 

\section*{Data availability}
 
Our manuscript has no associated data.

\end{document}